\documentclass[a4paper]{gtart}

\usepackage{graphicx}

\usepackage{amsmath, amssymb, latexsym, euscript}

\usepackage[inner=30mm, outer=30mm, textheight=225mm]{geometry}

\def\bkC{{\rm \kern.20em \vrule width.05em height1.5ex depth-.03ex \kern-.27em C}}
\def\C{\bkC}
\def\bksC{{\rm \kern.24em \vrule width.05em height1ex depth-.05ex 
\kern-.26em C}}

\def\bkE{{\rm I\kern-.22em E}}
 
\def\bkH{{\rm I\kern-.22em H}}
\def\H{\bkH} 
\def\bkN{{\rm I\kern-.17em N}}

\def\bkQ{{\rm \kern.24em \vrule width.05em height1.4ex depth-.05ex 
\kern-.26em Q}}

\def\bkR{{\rm I\kern-.17em R}}
\def\R{\bkR}

\def\bkZ{{\rm Z\kern-.28em Z}}

\DeclareMathOperator{\dev}{dev}
\DeclareMathOperator{\PGL}{PGL}
\DeclareMathOperator{\GL}{GL}
\DeclareMathOperator{\SL}{SL}
\DeclareMathOperator{\Mod}{Mod}
\DeclareMathOperator{\Alt}{Alt}

\DeclareMathOperator{\RP3}{\R P^3}
\DeclareMathOperator{\tr}{tr}

\DeclareMathOperator{\diag}{diag}

\theoremstyle{plain}
\newtheorem{theorem}{Theorem}%[section]
\newtheorem{lemma}[theorem]{Lemma}

\theoremstyle{definition}
\newtheorem{definition}[theorem]{Definition}
\newtheorem*{definition*}{Definition}

\theoremstyle{remark}
\newtheorem{remark}[theorem]{Remark}

\numberwithin{equation}{section}

\begin{document}

\title{Projective structures on a hyperbolic 3--orbifold}
\author{Joan Porti and Stephan Tillmann}

\begin{abstract} 
We compute and analyse the moduli space of those real projective structures on a hyperbolic 3--orbifold that are modelled on a single ideal tetrahedron in projective space. Parameterisations are given in terms of classical invariants, \emph{traces}, and geometric invariants, \emph{cross ratios}.
\end{abstract}

\makeshorttitle

%%%%%%%%%%%%%%%%%%%%%%%%%%%%%%%

\section{Introduction}

This note studies certain real projective structures on a 3-dimensional orbifold, 
$O,$ which is obtained by taking the one tetrahedron triangulation with two vertices of 
$S^3,$ deleting the vertices and modelling the edge neighbourhoods on $\R^3/\langle r \rangle,$ 
where $r$ is a rotation by $120^\circ$ (see Figure~\ref{fig:m137_deg}). This orbifold supports a 
unique complete hyperbolic structure; this has two Euclidean $(3,3,3)$--turnover cusps and is of finite volume.

In \cite{CLT0}, the philosophy was put forward that strictly convex projective manifolds behave 
like hyperbolic manifolds sans Mostow rigidity. This paper computes a moduli space of projective 
structures on $O$ that are modelled on an ideal tetrahedron. This moduli space, denoted $\Mod(O),$ 
turns out to be a smooth, open 2--dimensional disc.
We obtain two parameterisations of $\Mod(O)$: one in terms of algebraic invariants, \emph{traces}, and
one in terms of geometric invariants, \emph{cross-ratios.} 

The complete hyperbolic structure on $O$ is singled out as the only structure on $\Mod(O)$ having \emph{standard cusps}, 
whilst the remaining structures on $\Mod(O)$ all have \emph{generalised cusps}. It is also characterised as the unique 
fixed point of a natural involution on $\Mod(O).$ The problem to decide which of the structures on $\Mod(O)$ are \emph{properly convex} 
appears to be difficult by elementary means, but is completely solved by the theoretical results of Cooper, Long and Tillmann~\cite{CLT0, CLT1},
and  of Ballas, Danciger and Lee \cite{BallasDancigerLee}; 
the answer is \emph{all of them}. This problem has also motivated some of Choi's work~\cite{Choi2}.

% % To complete the discussion of the moduli space, we show that none of the structures carried by $D_1$ are properly or strictly convex, and we show that 
% % the unique fixed point of a natural involution on $D_1$ corresponds to an action of the alternating group $\Alt(5)$ on a tesselation of $\RP3$ by fifteen 3--simplices. This lifts to an action of the binary icosahedral group on the 3--sphere.

%%%%%%%%%%%%%%%%%%%%%%%%%%%

{\bf Acknowledgements.}
We thank the anonymous referee for suggestions that improved the paper. 
Research of the first author is supported by FEDER-MEC (grant number PGC2018-095998-B-I00).
Research of the second author is supported by an Australian Research Council Future Fellowship (project number FT170100316).

%%%%%%%%%%%%%%%%%%%%%%%%%%%

\section{Projective structures modelled on triangulations}

In $n$--dimensional real projective space, $n$--simplices are overly congruent. Given any two $n$--simplices, 
there is a projective transformation taking one to the other. Given an $n$--simplex, there is a $n$--dimensional 
family of projective transformations taking it to itself whilst fixing each of its vertices. The following notions 
can be defined in all dimensions, but we restrict to the case $n=3.$

The space $\RP3$ will be viewed as the set of 1--dimensional vector subspaces of $\R^4$ with the induced topology. 
The set of projective transformations, $\PGL(4, \R),$ then corresponds to the quotient of $\GL(4,\R)$ by its centre, 
the group of all non-zero multiples of the identity matrix. If $\Delta$ is the 3--simplex with vertices corresponding 
to the standard unit vectors $e_1,\ldots, e_4$ in $\R^4$ and containing $\sum e_i$ in its interior, then the family of 
projective transformations stabilising $\Delta$ and fixing its vertices corresponds to the set of diagonal matrices in 
$\GL(4,\R)$ having all entries positive or all entries negative. This gives a 3--dimensional family of projective transformations.

Let $M$ be an arbitrary, ideally triangulated 3--orbifold of finite topological type. We recall that an ideal
3--simplex is a 3--simplex with its four vertices removed, and an ideal triangulation of $M$ is 
 an expression of $M$ as a \emph{finite} collection of ideal 3--simplices with their faces glued in pairs, so that the 
branching locus of the orbifold is contained in the 2--skeleton. We assume that the interior of each ideal
$i$--simplex is embedded and that $M$ has finitely many ends, each
homeomorphic to the product of a 2-orbifold with a half line $[0+\infty)$.  The 
end-compactification $\overline{M}$  of $M$ is the result of adding the vertices of the ideal $3$--simplices,
topologically it is  the one-point compactification of each end of $M$.

A real projective structure on $M$ is (an equivalence class of) a pair $(\dev, \rho),$ where $\dev \co \widetilde{M} \to \RP3$ 
is a locally injective map and $\rho \co \pi_1(M) \to \PGL(4,\R)$ is a representation of the orbifold fundamental group which makes $\dev$ 
$\rho$-equivariant. 
Since the orbifold universal cover $\widetilde{M}$ is non-compact, some of its ends may be homeomorphic to $\R^2 \times (0,1).$ We therefore make some additional assumptions.

\emph{1. Structure is modelled on projective simplices:} 

Let $\widetilde{M}$ be the orbifold universal cover of $M$. Assume that $M$ is a good orbifold, so that
$\widetilde{M}$ is a manifold.
The ideal simplices of $M$ lift to
$\widetilde{M}$. Let  $\widehat{M}$ be the result of adding the vertices to the lifted simplices, with the corresponding identifications induced by side parings. 
For instance, for an ideally triangulated hyperbolic 3-orbifold of finite volume, $\widetilde{M}$ is the hyperbolic $3$-space  
$\H^3$ and 
$\widehat{M}$ consists in adding the set of points in the ideal boundary $\partial_{\infty} \H^3\cong S^2$
fixed by parabolic isometries, which is a dense countable subset of $\partial_{\infty} \H^3$.
In general, the set $\widehat{M}\setminus \widetilde{M}$ is a countable set of points
in which the triangulation is not locally finite. Furthermore, the 
end-compactification $\overline{M}$  of $M$ 
is the quotient of  $\widehat{M}$ by the action of the orbifold fundamental group of $M$.

% in bijection with the set of conjugacy classes  subgroups.

% 
% 
% $\overline{M}$ 

% 
% Let $\overline{M}$ be the end-compactification of $M$, namely the result of adding 
% 
% and $\widehat{M}$ be the universal cover of $\overline{M}.$ We may view $M \subset \overline{M}$ and $\widetilde{M} \subset \widehat{M},$ 
% and the complements consist of discrete sets of points. Lift the ideal triangulation of $M$ to an ideal triangulation of $\widetilde{M},$ 
% then each ideal 3--simplex in $\widetilde{M}$ corresponds to a 3--simplex in $\widehat{M}.$ 
% The induced map $\widehat{M} \to \overline{M}$ is simplicial.

\begin{center}
\emph{We assume that $\dev \co \widetilde{M} \to \RP3$ extends to a continuous, equivariant map $\widehat{\dev} \co \widehat{M} \to \RP3.$}
\end{center}

In this case, $\widehat{\dev}$ is $\rho_0$--equivariantly homotopic 
to a map $\widehat{\dev}_0$ with the property that $\widehat{\dev}_0(\Delta)$ 
is a projective simplex (of dimension 0, 1, 2 or 3) for every simplex $\Delta$ in $\widehat{M}$. 
In particular, the map is possibly not locally injective.

\emph{2. Structure is non-collapsed:} Let $\Delta$ a 3--simplex in $\widehat{M}.$ 
\begin{center}
\emph{We assume that the images of the vertices of $\Delta$ under $\widehat{\dev}$ are in general position.}
\end{center}

In this case, the above homotopy  between $\widehat{\dev}_0$ and $\widehat{\dev}$ can be assumed to preserve 
$\widetilde{M},$ and $\widehat{\dev}_0$ maps any simplex to a simplex of the same dimension. Moreover, it may be assumed to do so by a linear map and in particular, it is locally injective at interior points of 3--simplices.

\begin{definition}
Let $M$ be an ideally triangulated 3--orbifold with the property that the ideal triangulation restricts to an ideal triangulation of the singular locus. A \emph{real projective structure modelled on the triangulation} is a real projective structure $(\dev, \rho)$ on $M$ which is modelled on projective simplices and non-collapsed. If the triangulation consists of a single 3--simplex, we will also say that the structure is modelled on a 3--simplex.
\end{definition}

%%%%%%%%%%%%%%%%%%%%%%%%%%%

\section{The moduli space}

\begin{figure}[h]
  \begin{center}
    \includegraphics[height=3.6cm]{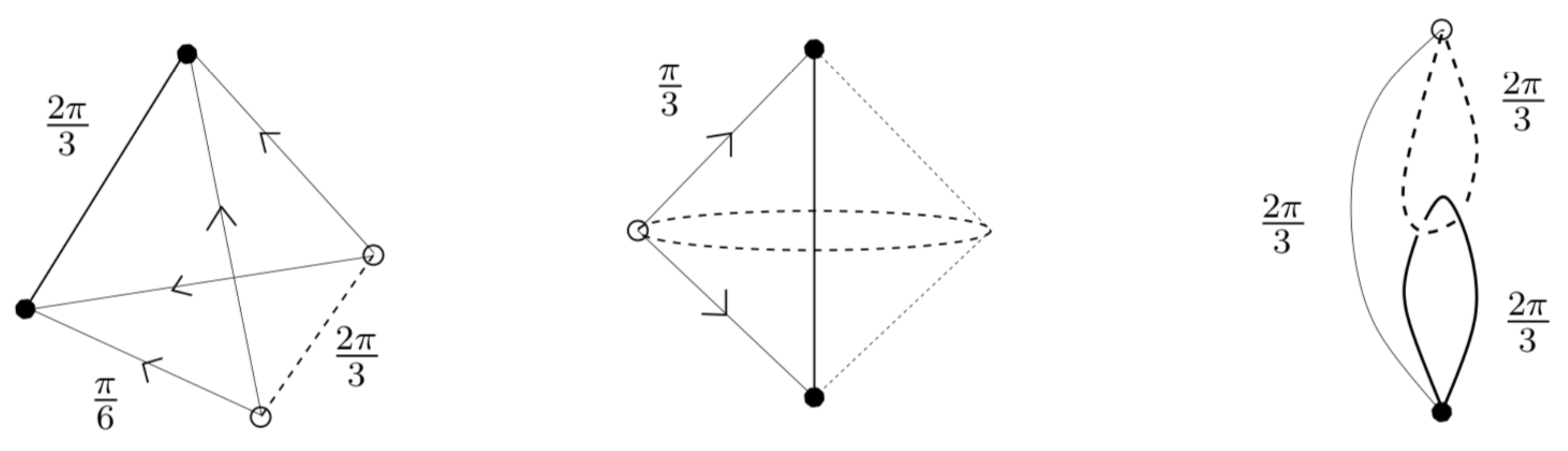}
  \end{center}
  \caption[Triangulation of $O$]{\small{To obtain $O,$ first glue the faces meeting along one of the edges with cone angle $2\pi/3$ to obtain a spindle, and then identify the boundary discs of the spindle. The result is $S^3$ minus two points, with the labelled graph (minus its vertices) as the singular locus. The hyperbolic structure can be obtained by identifying the ideal 3--simplex with an ideal hyperbolic 3--simplex with shape parameter $\frac{1}{2} + \frac{\sqrt{-3}}{6}.$ The fundamental group of $O$ admits, up to conjugation, exactly two irreducible representations into $\SL(2,\C).$ They are complex conjugates and correspond to holonomies for the hyperbolic structure.}}
  \label{fig:m137_deg}
\end{figure}

\begin{theorem}\label{thm:structures on O}
The set of real projective structures on $O$ modelled on a 3--simplex is parameterised by 
a connected component, $X$, 
of the set  of all $(w,x,y,z)\in \R^4$ subject to the following two equations:
\begin{align}\label{eqn: def rel in wxyz}
w+x+y+z&=3+wy,\\
\label{eqn: def rel in wxyz2}
wy&=zx.
\end{align}
The structures corresponding to any two distinct points of $X$ are neither 
isotopic nor projectively equivalent. 
Moreover, $X$ is diffeomorphic to an open disc. 

The
involution $(w,x,y,z) \to (y,z,w,x)$ 
on $X$ has exactly one fixed point, $(3,3,3,3) \in X$, that
corresponds to the complete hyperbolic structure on $O$. 
% % Furthermore, all structures are properly convex.  $X$ are  
% % 
% % and $(1,1,1,1) \in 
% % D_1.$ It hence restricts to an involution on each of the components. The fixed 
% % point $(3,3,3,3)$ corresponds to the complete hyperbolic structure on $O.$ 
% % Moreover, $D_0$ parameterises properly convex projective structures on $O.$ The 
% % fixed point $(1,1,1,1)$ corresponds to an action of the alternating group 
% % $\Alt(5)$ on  $\RP3$. Moreover, no point on 
% % $D_1$ corresponds to a properly or strictly convex projective structure on $O.$
\end{theorem}

Before proving the theorem, we describe the zero set of \eqref{eqn: def rel in wxyz} and \eqref{eqn: def rel in wxyz2}.

\begin{lemma}\label{lemma:components}
The zero set  of \eqref{eqn: def rel in wxyz} and \eqref{eqn: def rel in wxyz2} in $\R^4$
is a smooth surface with two components, each diffeomorphic to a disc. One of the components, $X$,
contains
the point with coordinates $w=x=y=z=3$, and it is distinguished from the other component by the inequalities
$x>1$ and $z>1$ (equivalently $y>1$ and $w>1$). 
\end{lemma}

\begin{proof}[Proof of the lemma]
From \eqref{eqn: def rel in wxyz} we replace $z= w\,y-w-x-y+3$ in \eqref{eqn: def rel in wxyz2}, so the set we want to describe is 
diffeomorphic
to
$$
\{(x,y,w)\in\R^3\mid -wxy+wx+wy+x^2+xy-3x=0\}.
$$
We change variables to simplify computations: set $w=a+1$, $x=b+1$, $y=c+1$. Then the subvariety becomes
$$
\{(a,b,c)\in\R^3\mid b^2-abc+a+c=0\}.
$$
We view the defining equation as a quadratic equation on $b$, whose discriminant is
$$
\mathrm{Disc}=a^2c^2-4 a-4c.
$$
It can be checked that the vanishing locus of $\mathrm{Disc}$ consist of precisely two smooth curves, each homeomorphic to $\R$
(for instance, by computing $c$ for each value of $a$, as $\mathrm{Disc}$ is quadratic on any of the variables). 
Also by elementary methods, one can show that  the set 
$\{(a,c)\in \R^2\mid \mathrm{Disc} \geq 0\}$ is homeomorphic
to two half planes.
Then the set we want to describe is the 2--to--1 branched covering of these half planes,
branched along their boundaries, and therefore the subvariety is diffeomorphic to two
smooth discs.
Furthermore the open quadrant $\{(a,b)\in\R^2\mid a> 0,\ c>0\}$ 
contains one of the components of $\mathrm{Disc}\geq 0$ and is disjoint from the other one. Finally,
the points  $(a,c)=(0,0)$ and $(a,c)=(2,2)$ lie in $\mathrm{Disc}=0$, but in different components.
\end{proof}

\begin{proof}[Proof of Theorem~\ref{thm:structures on O}]
Denote the ideal 3--simplex by $[v_1, v_2, v_3, v_4].$ 
The face pairings are $\alpha[v_1, v_2, v_3] =[v_1, v_2, v_4]$ and  $\beta[v_2, v_3, v_4] =[v_1, v_3, v_4].$ The orbifold fundamental group is generated by the face pairings, and we have the following presentation:
\begin{equation}
\pi_1^{orb}(O) = \langle \alpha, \beta : \alpha^3 = \beta^3 = (\alpha\beta\alpha^{-1}\beta^{-1})^3=1 \rangle.
\end{equation}

To determine all projective structures of $O$ modelled on a 3--simplex up to projective equivalence, 
it suffices to fix a projective 3--simplex, $\Delta$, and to 
determine all representations $\rho$ of $\pi_1(O)$ with $\rho(\alpha)$ and $\rho(\beta)$ as the corresponding face pairings. 
Here we have used that all 3--simplices are projectively equivalent. 
Choosing $\Delta= [e_1, e_2, e_3, e_4],$ the most general form of lifts of the face pairings is:
\begin{equation}
A= \begin{pmatrix} s_1 & 0 & 0 & a_1 \\
					0 & s_2 & 0 & a_2 \\
					0 & 0 & 0 & a_3\\
					0 & 0 & s_4 & a_4 \end{pmatrix} 
\quad\text{and}\quad
B = \begin{pmatrix} b_1 & t_1 & 0 & 0 \\
					b_2 & 0 & 0 & 0 \\
					b_3 & 0 & t_3 & 0\\
					b_4 & 0 & 0 & t_4 \end{pmatrix},
\end{equation}
subject to $s_1s_2s_4a_3\neq 0$ and $t_1t_3t_4b_2\neq 0.$

Since the above does not take division by the centre into account, the equation $A^3 = c I_4$ for $c \neq 0$ is projectively equivalent to $A^3 = I_4,$ since the equation $c^3 = 1$ always has a non-zero real root. Similarly for $B.$ This gives: 
\begin{equation}
A= \begin{pmatrix} 1 & 0 & 0 & a_1 \\
					0 & 1 & 0 & a_2 \\
					0 & 0 & 0 & a_3\\
					0 & 0 & -a_3^{-1} & -1 \end{pmatrix} 
\quad\text{and}\quad
B = \begin{pmatrix} -1 & -b_2^{-1} & 0 & 0 \\
					b_2 & 0 & 0 & 0 \\
					b_3 & 0 & 1 & 0\\
					b_4 & 0 & 0 & 1 \end{pmatrix},
\end{equation}
subject to $a_3b_2 \neq 0.$ Note that both matrices are elements of $\SL(4, \R).$ Since we are interested in representations up to conjugacy, one may conjugate the above to give:
\begin{equation}\label{eq:forms of A and B}
A= \begin{pmatrix} 1 & 0 & 0 & a_1 \\
					0 & 1 & 0 & a_2 \\
					0 & 0 & 0 & -1\\
					0 & 0 & 1 & -1 \end{pmatrix} 
\quad\text{and}\quad
B = \begin{pmatrix} -1 & 1 & 0 & 0 \\
					-1 & 0 & 0 & 0 \\
					b_3 & 0 & 1 & 0\\
					b_4 & 0 & 0 & 1 \end{pmatrix}.
\end{equation}
It now remains to analyse $(ABA^{-1}B^{-1})^3 = cI_4.$ With (\ref{eq:forms of A and B}), one first notes that $c=1.$ 
In particular, any representation of $\pi_1(O)$ into $\PGL(4, \R)$ lifts to a representation into $\SL(4,\R).$ 
By a case-by-case analysis, one obtains the following cases:

\emph{Case 1:} $a_1=a_2=b_3=b_4=0.$ 
In this case there is a single representation which in fact satisfies $A^3=B^3=ABA^{-1}B^{-1}=I_4.$ 
The corresponding developing map is not locally injective, and hence that there is no corresponding real projective structure.

\emph{Case 2:} One obtains a single equation:
\begin{equation}\label{eqn:def rel in ab}
(a_1 + a_2)(b_3+b_4) = 3 + a_1a_2b_3b_4.
\end{equation}
Analysis of which representations are conjugate yields that pairs $A,B$ and $A', B'$ give conjugate representations in $\PGL(4, \R)$ if and only if they are conjugate by $$M = \diag (m, m, m^{-1}, m^{-1})$$ for some $m \neq 0.$ The effect on the quadruples is: 
\begin{equation}\label{eqn: R^* action}
(a'_1, a'_2, b'_3, b'_4) = (ma_1, ma_2, m^{-1}b_3, m^{-1}b_4).
\end{equation}
Note that $M$ corresponds to a projective transformation stabilising any subsimplex of $\Delta.$ Since no face pairing is a reflection,
we need to check local injectivity at the edges of the triangulation. It follows from inspection that local injectivity at the axis 
of $A$ is equivalent to not both $a_1$ and $a_2$ to be contained in $(-\infty, 0]$, 
and local injectivity at the axis of $B$ is equivalent to not both $b_1$ and $b_2$ to be contained in $(-\infty, 0].$
Thus, the corresponding pair of equivariant map and representation, $(\dev, \rho),$ can be replaced by a projective 
structure given by $(M\circ \dev, M \circ \rho \circ M^{-1})$ with $M = \diag (m, m, m^{-1}, m^{-1})$ that is locally 
injective at both the axes of $A$ and $B$ unless ($a_1, a_2 \le 0$ and $b_1, b_2 \ge 0$) or ($a_1, a_2 \ge 0$ 
and $b_1, b_2 \le 0$). But equation (\ref{eqn:def rel in ab}) has no solutions of this form. 

It remains to check local injectivity around the axis of the commutator $C=ABA^{-1}B^{-1}$. 
The description  of $\dev$ around the axis of $C$ is obtained
by gluing translates of $\Delta$ following a cyclic order. 
Knowing that $\dev$ is locally injective on the edges of $A$ and $B$, 
we have that this process of gluing simplices around the axis of $C$ turns either once or several times:
namely either $\dev$ is locally injective or it is locally a cyclic branched covering, branched on the axis of $C$.
Moreover, since we require that $C^3=I_4$, by continuity  
the degree of this branched covering is locally constant along the parameter space. This degree is the criterion for choosing
one of the components of the parameter space and discarding the other.  
Before that,  we describe the parameter space, then we will check whether it is locally  
injective or a branched covering in a given point on each component.

The coordinates given in the statement of the theorem can be expressed in terms of classical invariants of the chosen lift of $\rho$:
\begin{align*}
w = a_1b_4 &= 2 + \tr AB,\\
x = a_1b_3 &= 2 + \tr A^{-1}B,\\
y = a_2b_3 &= 2 + \tr A^{-1}B^{-1},\\
z = a_2b_4 &= 2 + \tr AB^{-1},
\end{align*}
and (\ref{eqn:def rel in ab}) can be expressed in terms of these, giving the 
first equation given in (\ref{eqn: def rel in wxyz}). The second arises from 
$wy=a_1a_2b_3b_4=xz.$ It can now be verified that the zero set of \eqref{eqn: def rel in wxyz}
and \eqref{eqn: def rel in wxyz2} corresponds to the 
quotient of the action of $\R\setminus \{0\}$ given in (\ref{eqn: R^* action}) 
on the set of all $(a_1,a_2,b_3,b_4)$ subject to (\ref{eqn:def rel in ab}). 
We apply Lemma~\ref{lemma:components} to conclude that the parameter space has two components. 
Next we check wheter local injectivity of $\dev$ holds or not on a point on each component of this parameter space.

On the component $X$ we consider $(w,x,y,z)=(3,3,3,3)$: it is  
a straightforward computation that it corresponds to the holonomy of the hyperbolic structure.
In particular 
the devoloping map is locally injective along the axis of $C=ABA^{-1}B^{-1}$.

On the other component of the parameter space we consider 
$(w,x,y,z)=(1,1,1,1)$: 
we take $a_1=a_2=b_3=b_4=1$. We look at the cycle of consecutive translates of 
$\Delta$ than gives a full turn around
$e_1e_4$,
the axis of $C=ABA^{-1}B^{-1}$:
% $$
% \{ \Delta, A(\Delta), AB(\Delta), ABA^{-1}(\Delta),C(\Delta),CA(\Delta), CAB(\Delta), \dotsc, BA(\Delta), B(\Delta) \}
% $$
\begin{multline}
\label{eqn:cycle}
( \Delta, A(\Delta), AB(\Delta), ABA^{-1}(\Delta),C(\Delta),CA(\Delta), CAB(\Delta), 
CABA^{-1}(\Delta), 
\\
C^2(\Delta), BAB^{-1}(\Delta) , BA(\Delta), B(\Delta) ).
\end{multline}
This cycle has 12 elements, because  $C$ has order 3 and the edge has four representatives in $\Delta$. 
The matrix $C$ acts on this cycle by a cyclic permutation. 
By direct computation, when $a_1=a_2=b_3=b_4=1$:
$$
CAB=
\begin{pmatrix}
 0 & 1 & 0 & 0 \\
 1 & 0 & 0 & 0 \\
 0 & 0 & 0 & 1 \\
 0 & 0 & 1 & 0 
\end{pmatrix}
$$
Hence $CAB(\Delta)=\Delta$ and the cycle \eqref{eqn:cycle} turns at least twice around the edge, 
i.e.~$\dev$ is locally a branched covering of degree $\geq 2$. (We prove in the Remark~\ref{remark:1111} that the 
degree is precisely~2.) Hence we discard the whole component.

Let $X$ be the component of the parameter space containing $(w,x,y,z)=(3,3,3,3)$. We have shown that it parameterizes 
projective structures, 
we aim to show that these structures  are properly convex. For that purpose, 
% % 
% % to apply the results of \cite{BallasDancigerLee} and \cite{CLT1} on deformations of properly convex structures, 
 we look at the peripheral subgroups
of $\pi_1^{orb}(O)$. There are two conjugacy classes of peripheral subgroups, corresponding
to the ends of $O$ and represented respectively by the stabilizer of $v_1$ and the stabilizer of $v_4$. 
The stabilizer of $v_1$ is the group generated by $\alpha$ and $\gamma=[\alpha,\beta]$, it is the 
fundamental group of the link of the end of $O$ corresponding to
$v_1$ (which is a turnover $S^2(3,3,3)$): 
$$
\pi_1^{orb} ( S^2(3,3,3) )\cong\langle \alpha, \gamma\mid \alpha^3=\gamma^3=(\alpha^2 \gamma)^3=1\rangle.$$
Similarly, for $v_4$ the peripheral group is generated by $\beta$ and $\gamma$.
The maximal torsion-free subgroup  of the peripheral group of $v_1$ has index 3, it
is generated by $\alpha\gamma$ and $\gamma\alpha$ and it is isomorphic to $\mathbb Z^2$.

\begin{lemma}
\label{lemma:hexends}
For $(w,x,y,z)\in X\setminus\{(3,3,3,3) \}$, up to conjugation we have
$$
A=\rho(\alpha)=
\begin{pmatrix}
 0 & 0 & 1& 0 \\
 1 & 0 & 0 & 0 \\
 0 & 1 & 0 & 0 \\
 0 & 0 & 0 & 1
\end{pmatrix}
\qquad
\textrm{ and }\qquad
AC=\rho(\alpha\gamma)=
\begin{pmatrix}
 \lambda_1 & 0 & 0 & 0 \\
 0 & \lambda_2 & 0 & 0 \\
 0 & 0 & \lambda_3 & 0  \\
 0 & 0 & 0 & 1
\end{pmatrix}
$$
with $\lambda_i\in \R$, $\lambda_i>0$  and  $\lambda_1\neq \pm 1$. 
\end{lemma}

Assuming the lemma, we prove that the set of properly convex projective structures on $X$ is open and closed.
For openess,   the 
holonomy of $\langle \alpha\gamma,\gamma\alpha\rangle\cong\mathbb{Z}^2$
 preserves a flag 
 in $\mathbb R^4$, for any structure on $X$. We can argue similarly for $v_4$
and the peripheral subgroup generated by $\beta$ and $\gamma$. Hence, we may apply Theorem~0.2 of \cite{CLT1} to say that set of 
properly convex structures is open in $X$.
For closedness, 
we apply Theorem~1.5 of \cite{BallasDancigerLee}. Namely, 
the compact orbifolds obtained by truncating the ends have totally geodesic boundary (two turnovers).
% the ends of these projective structures are totally geodesic.
Ballas, Danciger and Lee show in \cite{BallasDancigerLee} that 
 these structures
yield a convex structure on the double of the truncated orbifolds, and  apply a theorem of Benoist for 
convex structures on closed orbifolds
\cite{BenoistIII}. 
Alternatively, one can directly refer to Theorem 0.1 of the forthcoming work~\cite{CT2020} to obtain closedness.
Thus every structure on $X$ is properly convex. 
% 
% 
% 
% Next we prove that no structure on $D_1$ is properly convex. When $(w,x,y,z)=(1,1,1,1)$, 
% the group generated by $A$ and $B$ is  $\Alt(5)$, in particular it is finite
% and 
% the projective structure cannot be properly convex. 
% When $(w,x,y,z)\in D_1\setminus\{(1,1,1,1)\}$ we argue by contradiction: assume that $\langle A, B\rangle$ acts properly discontinuosly 
% on a convex domain $\Omega\subset\R^3\subset \RP3$, properly contained in an affine chart. Let $p_1,p_2,p_3\in\RP3$ 
% correspond to the first three invariant subspaces of $\rho(\alpha\gamma)$ in Lemma~\ref{lemma:hexends},
% in particular they are permuted by $\rho(\alpha)$.
% As $\lambda_1\neq 1$  and taking into account that $ \rho(\alpha)$ permutes them, 
% $p_1,p_2,p_3 \in \overline\Omega$.
% By convexity, $\overline\Omega$ contains one of the 4 totally geodesic triangles 
% with vertices $p_1$, $p_2$, and $p_3$.
% The closure of the union of these four triangles is the  projective plane
% $\operatorname{\R P^2}$ spann by $p_1$, $p_2$, and $p_3$. As $\lambda_1<0$ but $\lambda_3>0$,
% $\rho(\alpha)$ and  $\rho(\alpha\gamma)$ permute transitivelly those four triangles. 
% Hence $\operatorname{\R P^2}\subset  \overline \Omega$, which is a contradiction.
\end{proof}
 
\begin{proof}[Proof of Lemma~\ref{lemma:hexends}]
First we claim that all eigenvalues of $AC=\rho(\alpha\gamma)$ are real, for each value of the parameters $(w,x,y,z)\in X$. 
For that purpose, we compute the characteristic polynomial of $\rho(\alpha\gamma)=AC$, by a long but explicit computation:
$$
(\lambda-1)({\lambda}^{3}+ \left( -yx-w+2\,x+2\,y-z \right) {\lambda}^{2}+ \left( 
zw-2\,w+x+y-2\,z \right) \lambda-1)
$$
and we check that the discriminant of its degree three factor is nonnegative. 
To check that, we write $w$ and $z$ in terms of $x$ and $y$ from \eqref{eqn: def rel in wxyz} and \eqref{eqn: def rel in wxyz2};
this only parameterizes an open dense subset of  $X$,
but it is sufficient to determine the non-negativity of the discriminant. With this parameterization of a subset $X$, the discriminant is
$$
{\frac { \left( {y}^{2}-3\,y+3 \right) ^{2} \left( {x}^{2}-3\,x+3
 \right) ^{2} \left( -y+x \right) ^{2} \left( {x}^{2}{y}^{2}-3\,{x}^{2
}y-3\,x{y}^{2}+3\,{x}^{2}+3\,xy+3\,{y}^{2} \right) ^{2}}{ \left( xy-x-
y \right) ^{6}}},
$$
which is always nonnegative.

From the formula of the characteristic polynomial of $AC=\rho(\alpha\gamma)$, 
an elementary computation shows that for $(w,x,y,z)\in X\setminus\{(3,3,3,3)\}$, 
$AC$ has a (real) eigenvalue $\lambda_1\neq 1$. 
Up to replacing it by $1/\lambda_1^2$, we assume that $\lambda_1$ has multiplicity $1$.
Let $p_1\in\RP3$ be the projective point
fixed by 
$AC$ corresponding to $E_{\lambda_1}(AC)$,  the eigenspace of eigenvalue $\lambda_1$.
Set  $p_2=A(p_1)$ and $p_3=A(p_2)$,
we have $p_1=A(p_3)$.
By construction, 
\begin{equation}\label{eqn:eigenspaces}
p_1= E_{\lambda_1}(AC),\quad
p_2=E_{\lambda_1}(A^2CA^2),\quad
p_3=E_{\lambda_1}(CA),\quad
\end{equation}
where $E_{\lambda}$ denotes the eigenspace with eigenvalue $\lambda$.
Next we aim to show that $p_1$, $p_2$ and $p_3$ span a 3-dimensional linear subspace. 
We first show that $p_1\neq p_2$. 
Seeking a contradiction, assume $p_1=p_2$, then also $p_3=p_2=p_1$. As the product of the matrices $CA$, $A^2CA^2$, $AC$  is $1$,
\eqref{eqn:eigenspaces} would imply that $\lambda_1^3= 1$, which contradicts $\lambda_1\neq1$ and $\lambda_1\in\R$.
Let $V\subset\R^4$ be the span of the three affine lines $p_1$, $p_2$ and $p_3$. 
Commutativity between $AC$
and $CA$ implies that $AC(p_2)=p_2$ and $AC(p_3)=p_3$. Thus $V$ contains three different lines that are invariant by $AC$.
If $V$ had dimension $2$, then this would imply that  $AC$ acts as a multiple of the identity on $V$, contradicting that
$\lambda_1$ has multiplicity 1. Hence $\dim V=3$, as claimed.
% % 
% % 
% % 
% % Of course $\rho(\alpha)$ permutes the three 
% % projective points $p_1$, $p_2$ and $p_3$ but we haven't shown yet that they are different. More precisely,
% % we aim to show that $p_1$, $p_2$ and $p_3$ are eigenspaces of $\rho(\alpha\gamma)$ 
% % in general position.  By construction 
% % $p_2$ is the eigenspace of $\rho(\alpha^2\gamma\alpha^2)$ sithe eigenvector $\lambda_1$,
% % and 
% % $p_3$ is the eigenspace of  
% % $\rho(\gamma\alpha)$ with eigenvector $\lambda_1$ (in each case $\lambda_1$ has multiplicity 1). 
% % As $\alpha\gamma$ 
% % and $\gamma\alpha$ commute,  $\rho(\alpha\gamma) (p_3)=p_3$, and by a similar argument 
% % $\rho(\alpha\gamma) (p_2)=p_2$. Next we show that $\rho(\alpha)$ does not fix the projective point $p_1$. 
% %  By contradiction,
% % if $\rho(\alpha)(p_1)=p_1$, then the three matrices $\rho(\alpha\gamma)$, $\rho(\gamma\alpha)$ 
% % and $\rho(\alpha^2\gamma\alpha^2)$ fix $p_1$ with the same eigenvalue $\lambda_1$. As the product of the three elements
% % $\gamma\alpha$, $\alpha\gamma$ and $\alpha^2\gamma\alpha^2$ is trivial, $\lambda_1^3=1$, contradicting that $\lambda_1\neq 1$.
% % In addition, the three affine lines $p_1$, $p_2$ and $p_3$ are in general position, because they are fixed by 
% % $\rho(\alpha\gamma)$ and $\lambda_1$ has multiplicity $1$.
 As $\dim (\R^4/V)=1$  and 
$A$ and $C$ have order three, both $A$ and $C$ act trivially on $\R^4/V$.
Thus every matrix in $\langle A, C \rangle$ has an eigenvector equal to 1. 
By looking at the action of $\langle A, C\rangle$ on  the vector subspace $V$,
one can find a matrix in $\langle AC,CA\rangle\cong\mathbb Z^2$ whose restriction to $V$
has all eigenvalues  different from $1$.
This gives the fourth vector for a basis of $\R^4$ so that the matrices with respect to these basis are as claimed.

Finally, the assertion on the sign of the eigenvalues follows from continuity by looking at the representation at 
$(w,x,y,z)=(3,3,3,3)$.
\end{proof}

\begin{remark}\label{remark:1111}
When  $(w,x,y,z)=(1,1,1,1)$,  $\langle A, B\rangle$ acts as $\Alt(5)$ by permuting five points in $\RP3$.
More precisely,  we take
 $a_1=a_2=b_3=b_4=1$, and  both $ A$ and $B$ preserve the set of five projective points
$
\{ \langle e_1\rangle , \langle e_2 \rangle , \langle e_3 \rangle ,\langle  e_ 4 \rangle ,\langle  e_1+ e_2-e_3-e_ 4\rangle  \}
$ in $\RP3$.
The action on this set determines the group action, as a projective transformation of  $\RP3$ is determined by the action on 
five points in general position. 
As both $A$ and $B$ act as cycles of order three on this set, and their fixed points are disjoint, this is the action of the alternating
group $\Alt(5)$.
It can be shown that these five points are the vertices of a tessellation of 
$\RP3$  defined by the translates of $\Delta$. 
The stabilizer of $\Delta $ has $4$ elements and the tessellation consists of 15 tetrahedra 
(each tetrahedron shares its vertices with two other tetrahedra of
the   tessellation). 
Furthermore, one can check that the edges invariant by $A$ or $B$ are adjacent to three simplices, and the edge invariant by $C$ 
is adjacent to six simplices (and therefore the developing map has degree $2=12/6$ around the edge of $C$).

The quotient of  $\RP3$ by this action of $A_5$ is the spherical orbifold $J\times^*_J J$ in \cite{Dunbar94}.
\end{remark}

\begin{remark}
By the discussion in the proof of Theorem~\ref{thm:structures on O} and Remark~\ref{remark:1111},
the component of the parameter space described by \eqref{eqn: def rel in wxyz} and \eqref{eqn: def rel in wxyz2} 
that contains $(w,x,y,z)=(1,1,1,1)$ 
corresponds to projective structures \emph{branched} along a singular edge, as described by Ballas and Casella in~\cite{BallasCasella}.
\end{remark}

\begin{remark}
The  proof of Theorem~\ref{thm:structures on O} gives an interpretation of the 
coordinates in terms of traces. We wish to point out that there is an alternative interpretation in terms of \emph{cross ratios} as follows.
We look for cross ratios of points in the projective line  
$$
l=\langle e_3,e_4\rangle $$ 
that contains the edge $e_3e_4$, namely the fixed point set of $B$.
We consider the projective plane 
$$
\Pi=\langle e_1, e_2, e_3\rangle,
$$
so that both $\Pi$ and   $A(\Pi)$ contain a face of the 3-simplex $\Delta$. The first three points of  $l$ that we consider
are the respective  intersection  of $l$ with $\Pi$, $A(\Pi)$ and  $A^2(\Pi)$. Their
 homogeneous coordinates are
$$
p_3=\Pi\cap l=[0:0:1:0],\quad p_4=A(\Pi)\cap l=[0:0:0:1] , \ \textrm{ and }\ p_5=A^2(\Pi)\cap l= [0:0:1:1].
$$
Each coordinate $x$, $y$, $z$ and $w$ appears as a cross ratio of $p_3$, $p_4$ and $p_5$ with one of the following:
\begin{align*}
  p_x & = A B^{-1}A(\Pi) \cap  l  =[0:0:1: 1-x]\\ 
  p_y & = ABA(\Pi)\cap l = [0:0: 1: 1-y], \\
  p_z & = A^{-1}B(\Pi)\cap l = [0:0:1-z:1], \\
  p_w & =  A^{-1}B^{-1}(\Pi)\cap l  = [0:0:1-w:1].
\end{align*}
Namely: 
\begin{align*}
x&= ( p_x,p_3; p_5,p_4) ,\\
y&= ( p_y,p_3; p_5,p_4), \\
z&= ( p_z,p_4; p_5,p_3) ,\\
w&= ( p_w,p_4; p_5,p_3) .
\end{align*}
\end{remark}  

\begin{remark}
Ballas and Casella~\cite{BallasCasella} study the same example but, instead of properly convex projective structures, 
they consider structures equipped with peripheral flags. Their deformation space is just a point, in contrast with this setting,
because for the generic structures the peripheral subgroups do not have invariant flags.
\end{remark}

%%%%%%%%%%%%%%%%%%%%%%%%%%%

\small
% \bibliography{projective.bib} 
% \bibliographystyle{plain} 

%%%%%%%%%%%%%%%%%%%%%%%%%%%%

\address{Departament de Matem\`atiques, Universitat Aut\`onoma de Barcelona and CRM, 08193 Bellaterra, Spain } 
\email{porti@mat.uab.cat} 

\address{School of Mathematics and Statistics, The University of Sydney, NSW 2006, Australia} 
\email{stephan.tillmann@sydney.edu.au} 
\Addresses
                        
\end{document}